\documentclass{amsart}
\usepackage[english]{babel}
\usepackage{amsmath,amssymb,amsthm}
\usepackage{bbm}
\usepackage{color}
\newtheorem{theorem}{Theorem}[section]
\newtheorem{corollary}[theorem]{Corollary}

\newtheorem{lemma}[theorem]{Lemma}
\newtheorem{proposition}[theorem]{Proposition}

\begin{document}

\title[Derivations with values in quasi-normed bimodules]{Derivations with values in quasi-normed bimodules of locally measurable operators}
\author{A. F. Ber}
\address{Department of Mathematics,National University of Uzbekistan,
Vuzgorodok, 100174, Tashkent, Uzbekistan}
\email{ber@ucd.uz}

\author{V. I. Chilin }
\address{Department of Mathematics, National University of Uzbekistan,
Vuzgorodok, 100174, Tashkent, Uzbekistan}
\email{chilin@ucd.uz}

\author{G. B. Levitina }
\address{Department of Mathematics, National University of Uzbekistan,
Vuzgorodok, 100174, Tashkent, Uzbekistan}
\email{bob\_galina@mail.ru}


\begin{abstract}
We prove that every derivation acting on a von Neumann algebra $\mathcal{M}$ with values in a quasi-normed bimodule of locally measurable operators affiliated with $\mathcal{M}$ is necessarily inner.
\end{abstract}

\maketitle
\date{\today}

\section{Introduction}
One of the important results in the theory of derivations in
Banach bimodules is the Theorem of J. R. Ringrose on automatic
continuity of every derivation acting on a $C^*$-algebra $\mathcal{M}$
with values in a Banach $\mathcal{M}$-bimodule \cite{Ringrose}. This theorem
extends the well-known result that every derivation of a
$C^*$-algebra $\mathcal{M}$ is automatically norm continuous
\cite{Sak}. In the case when $\mathcal{M}$ is an $AW^*$-algebra (in
particular, $W^*$-algebra), every derivation on $\mathcal{M}$ is
inner \cite{Olesen}, \cite{Sak}.

Significant examples of $W^*$-modules are non-commutative
symmetric spaces of measurable operators affiliated with a von Neumann algebra. At the present time the theory of symmetric spaces is actively developed (see, e.g.
\cite{DDdP}, \cite{K-S}) and it gives useful applications both in the geometry of Banach spaces and in the theory of unbounded operators. Every non-commutative symmetric space
is a  solid linear space in the $*$-algebra $S(\mathcal{M},\tau)$
of all $\tau$-measurable operators affiliated with a von Neumann algebra $\mathcal{M}$, where $\tau$ is a faithful normal semifinite trace on $\mathcal{M}$ \cite{Ne}. The algebra $S(\mathcal{M},\tau)$ equipped with the natural topology  $t_\tau$ of convergence in measure generated by the trace  $\tau$ is a complete metrizable topological algebra. In its turn the algebra $S(\mathcal{M},\tau)$ represents a solid $*$-subalgebra of the $*$-algebra $LS(\mathcal{M})$ of
all locally measurable operators, affiliated with a von Neumann
algebra $\mathcal{M}$ \cite{San}, \cite{Yead}, in addition, the latter algebra $LS(\mathcal{M})$ with the natural topology  $t(\mathcal{M})$ of convergence locally in measure is a complete topological $*$-algebra \cite{Yead}.

In \cite{AAK,AK,BCS_2,BCS_3} there are established significant results in description of derivations of these algebras. In particular, it is proved that for every
$t(\mathcal{M})$-continuous derivation $\delta$ acting on the algebra $LS(\mathcal{M})$
 there exists  $a\in LS(\mathcal{M})$, such
that $\delta(x) = \delta_a(x) = ax-xa=[a,x]$ for all $x\in LS(\mathcal{M})$, that is the
derivation $\delta$ is inner. One of the main corollary of this result provides a full description of  derivations $\delta$ acting from $\mathcal{M}$ into a Banach $\mathcal{M}$-bimodule $\mathcal{E}$ of locally measurable operators affiliated with $\mathcal{M}$. More precisely it is shown that any derivation an a von Neumann algebra $\mathcal{M}$ with values in a Banach $\mathcal{M}$-bimodule $\mathcal{E}$ is inner \cite{BCS_3}.

In the present paper we establish the similar property of derivations in more general case when $\mathcal{E}$ is a quasi-normed $\mathcal{M}$-bimodule of locally measurable operators affiliated with $\mathcal{M}$.
The proof proceed in several stages. Firstly, in Section 2 we establish the version of Ringrose Theorem for arbitrary quasi-normed $C^*$-bimodules without the assumption of their completeness. After that in section 3 we supply the proof
of the main result of the present paper (Theorem \ref{t_mbm}) showing that every  derivation $\delta:\mathcal{M}\rightarrow \mathcal{E}$ is necessarily inner, i.e.  $\delta = \delta_d$ for some $d \in \mathcal{E}$.
In particular, $\delta$ is a continuous
derivation from $(\mathcal{M},\|\cdot\|_{\mathcal{M}})$ into
$(\mathcal{E},\|\cdot\|_{\mathcal{E}})$. In addition, the operator
$d\in\mathcal{E}$ may be chosen so
that  $\|d\|_{\mathcal{E}}\leq
2C_\mathcal{E}\|\delta\|_{\mathcal{M}\rightarrow\mathcal{E}}$, where $C_\mathcal{E}$ is the modulus of concavity of the quasi-norm $\|\cdot\|_\mathcal{E}$.

We use terminology and notations from the von Neumann algebra
theory \cite{KR,Tak} and the theory of locally measurable
operators from \cite{MCh,Yead}.

\section{Preliminaries}

Let $B(H)$ be the $*$-algebra of all
bounded linear operators acting in  a Hilbert space $H$, and let $\mathbf{1}$ be the
identity operator on $H$. Let $\mathcal{M}$ be a von Neumann algebra
 acting on $H$, and let $\mathcal{Z}(\mathcal{M})$ be
the centre of $\mathcal{M}$. Denote by
$\mathcal{P}(\mathcal{M})=\{p\in\mathcal{M}:\ p=p^2=p^*\}$ the
lattice of all projections in $\mathcal{M}$ and by
$\mathcal{P}_{fin}(\mathcal{M})$  the set of all finite projections in
$\mathcal{M}$.

A densely-defined closed linear operator $x$
affiliated with $\mathcal{M}$  is said to be \emph{measurable}
with respect to $\mathcal{M}$ if there exists a sequence
$\{p_n\}_{n=1}^\infty\subset \mathcal{P}(\mathcal{M})$ such that
$p_n\uparrow \mathbf{1},\ p_n(H)\subset \mathfrak{D}(x)$ and
$p_n^\bot=\mathbf{1}-p_n\in \mathcal{P}_{fin}(\mathcal{M})$ for every
$n\in\mathbb{N}$, where $\mathfrak{D}(x)$ is the domain of $x$ and
$\mathbb{N}$ is the set of all natural numbers. The set
$S(\mathcal{M})$ of all measurable operators is a unital
$*$-algebra over the field $\mathbb{C}$ of complex numbers with respect to strong sum $\overline{x+y}$, strong product $\overline{xy}$ and the adjoint operation $x^*$ \cite{Seg}.

A densely-defined closed linear operator $x$ affiliated with
$\mathcal{M}$ is called \emph{locally measurable} with respect to
$\mathcal{M}$ if there is a sequence $\{z_n\}_{n=1}^\infty \subset \mathcal{P}(\mathcal{Z}(\mathcal{M}))$
 such that $z_n \uparrow
\mathbf{1},\ z_n(H)\subset\mathfrak{D}(x)$ and $xz_n\in
S(\mathcal{M})$ for all $n\in\mathbb{N}$.

The set $LS(\mathcal{M})$ of all locally measurable operators
(with respect to $\mathcal{M}$) is a unital $*$-algebra over the
field $\mathbb{C}$ with respect to the same algebraic operations as in $S(\mathcal{M})$, in addition  $\mathcal{M}$ and $S(\mathcal{M})$ are
$*$-subalgebras of $LS(\mathcal{M})$.

For every subset $E\subset LS(\mathcal{M})$, the set of all
self-adjoint (respectively, positive) operators in $E$ is denoted by
$E_h$ (respectively, $E_+$). The partial order in $LS_h(\mathcal{M})$ is
defined by its cone $LS_+(\mathcal{M})$ and is denoted by $\leq$.

Let $x$ be a closed linear operator with the dense domain $\mathfrak{D}(x)$
in $H$, let $x=u|x|$ be the polar decomposition of the operator
$x$, where $|x|=(x^*x)^{\frac{1}{2}}$ and $u$ is a  partial
isometry in $B(H)$ such that $u^*u$ is the right support $r(x)$ of
$x$.  It is known that $x\in LS(\mathcal{M})$ (respectively, $x\in
S(\mathcal{M})$) if and only if $|x|\in LS(\mathcal{M})$
(respectively, $|x|\in S(\mathcal{M})$) and $u\in
\mathcal{M}$~\cite[\S\S\,2.2, 2.3]{MCh}. For every $x\in LS(\mathcal{M})$ the projection  $s(x)=l(x)\vee r(x)$, where $l(x)$ and $r(x)$ are left and right supports of $x$ respectively, is called the support of $x$. If $x \in LS_h(\mathcal{M})$, then the spectral family
of projections $\{E_\lambda(x)\}_{\lambda\in
  \mathbb{R}}$ for $x$ belongs to
$\mathcal{M}$~\cite[\S\,2.1]{MCh}, in addition, $s(x) = l(x) = r(x)$.

Denote by $\|\cdot\|_\mathcal{M}$ the $C^*$-norm in the von Neumann algebra $\mathcal{M}$. We need the following property of partial order in the algebra $LS(\mathcal{M})$.

\begin{proposition}\cite[Proposition 6.1]{BCS_3}
\label{mchext} Let $\mathcal{M}$ be a von Neumann algebra acting on the Hilbert space $H$, $x,y\in LS_+(\mathcal{M})$ and $y\leq x$. Then $y^{1/2}=ax^{1/2}$ for some $a\in s(x)\mathcal{M}s(x),\ \|a\|_{\mathcal{M}}\leq 1$, in particular, $y=axa^*$.
\end{proposition}

Now, let us recall the definition of the local measure topology.
If $\mathcal{M}$ is a commutative von Neumann algebra, then
$\mathcal{M}$ is $*$-isomorphic to the $*$-algebra
$L^\infty(\Omega,\Sigma,\mu)$ of all essentially bounded
measurable complex-valued functions defined on a measure space
$(\Omega,\Sigma,\mu)$ with the measure $\mu$ satisfying the direct
sum property (we identify functions that are equal almost
everywhere) (see e.g. \cite[Ch. III, \S 1]{Tak}).
The direct sum property of a measure $\mu$ means
that the Boolean algebra $\mathcal{P}(L^\infty(\Omega,\Sigma,\mu))$  of all projections of a commutative von Neumann algebra
$L^\infty(\Omega,\Sigma,\mu)$ is order complete, and for any
non-zero $\chi_E  \in \mathcal{P}(L^\infty(\Omega,\Sigma,\mu))$ there exists a non-zero
projection $\chi_F \leq \chi_E$ such that $\mu(F)<\infty$, where $E, F \in\Sigma$, $\chi_E(\omega)=1,\ \omega\in E$
and $\chi_E(\omega)=0$, when $\omega\notin E$.

Consider the $*$-algebra
$LS(\mathcal{M})=S(\mathcal{M})=L^0(\Omega,\Sigma,\mu)$ of all
measurable almost everywhere finite complex-valued functions
defined on $(\Omega,\Sigma,\mu)$ (functions that are equal almost
everywhere are identified).  Define on $L^0(\Omega,\Sigma,\mu)$
the local measure topology $t(L^\infty(\Omega))$, that is, the
Hausdorff vector topology, whose base of neighbourhoods of zero is
given by
$$
  W(B,\varepsilon,\delta):= \{f\in\ L^0(\Omega,\, \Sigma,\, \mu)
  \colon
  \ \hbox{there exists a set} \ E\in \Sigma\
  \mbox{such that}
  $$
  $$
   E\subseteq B, \ \mu(B\setminus
  E)\leq\delta, \ f\chi_E \in L^\infty(\Omega,\Sigma,\mu), \
  \|f\chi_E\|_{{L^\infty}(\Omega,\Sigma,\mu)}\leq\varepsilon\},
$$
where $\varepsilon, \ \delta >0$, $B\in\Sigma$, $\mu(B)<\infty$.

The topology
$t(L^\infty(\Omega))$ does not change if the measure $\mu$ is
replaced with an equivalent measure \cite{Yead}.

Now let $\mathcal{M}$ be an arbitrary von Neumann algebra and let
$\varphi$ be a $*$-isomorphism from $\mathcal{Z}(\mathcal{M})$
onto the $*$-algebra $L^\infty(\Omega,\Sigma,\mu)$.  Denote by
$L^+(\Omega,\, \Sigma,\, m)$ the set of all measurable real-valued
functions defined on $(\Omega,\Sigma,\mu)$ and taking values in
the extended half-line $[0,\, \infty]$ (functions that are equal
almost everywhere are identified).
Let $\mathcal{D}\colon \mathcal{P}(\mathcal{M})\to
L^+(\Omega,\Sigma,\mu)$ be  a dimensional function on $\mathcal{P}(\mathcal{M})$ \cite{Seg}.

For arbitrary scalars $\varepsilon , \delta >0$ and a set $B\in
\Sigma$, $\mu(B)<\infty$, we set

 $$ V(B,\varepsilon, \delta ) := \{x\in LS(\mathcal{M})\colon
  \mbox{there exist} \ p\in \mathcal{P}(\mathcal{M}),
  z\in \mathcal{P}(\mathcal{Z}(\mathcal{M})),
    \mbox{such that}$$
    $$ xp\in \mathcal{M},
  \|xp\|_{\mathcal{M}}\leq\varepsilon,
  \ \varphi(z^\bot) \in W(B,\varepsilon,\delta), \
    \mathcal{D}(zp^\bot)\leq\varepsilon \varphi(z)\}.$$

It was shown in~\cite{Yead} that the system of sets
\begin{equation}
\label{eq_xV}
 \{x+V(B,\,\varepsilon,\,\delta)\colon \ x \in LS(\mathcal{M}),\
 \varepsilon, \ \delta >0,\ B\in\Sigma,\ \mu(B)<\infty\}
\end{equation}
defines a Hausdorff vector topology $t(\mathcal{M})$ on
$LS(\mathcal{M})$ such that the sets of \eqref{eq_xV} form a neighbourhood base of an
operator $x\in LS(\mathcal{M})$. The topology $t(\mathcal{M})$ on $LS(\mathcal{M})$ is called the \textit{local
measure topology}.
It is known that
$(LS(\mathcal{M}), t(\mathcal{M}))$ is a complete topological
$*$-algebra, and the topology $t(\mathcal{M})$ does not depend on
a choice of dimension function $\mathcal{D}$~  and on a choice
of $*$-isomorphism $\varphi$ (see e.g. \cite[\S3.5]{MCh},
\cite{Yead}).

Let us mention the following important property of the topology $t(\mathcal{M})$.

\begin{proposition}\cite[Proposition 2.5]{BCS_3}
\label{p2_2}
The von Neumann algebra $\mathcal{M}$ is dense in  $(LS(\mathcal{M}),t(\mathcal{M}))$.
\end{proposition}

Let $\mathcal{M}$ be a semifinite von Neumann algebra acting on the
Hilbert space $H$, let $\tau$ be a faithful normal semifinite trace on
$\mathcal{M}$. An operator $x\in S(\mathcal{M})$ is called \emph{$\tau$-measurable} if for any
$\varepsilon>0$ there exists a projection
$p\in\mathcal{P}(\mathcal{M})$ such that $p(H)\subset
\mathfrak{D}(x)$ and $\tau(p^\bot)<\varepsilon$.

The set $S(\mathcal{M},\tau)$ of all $\tau$-measurable operators
is a  $*$-subalgebra of $S(\mathcal{M})$ and $\mathcal{M}\subset S(\mathcal{M},\tau)$. If the trace $\tau$ is finite, then
$S(\mathcal{M},\tau)=S(\mathcal{M})=LS(\mathcal{M})$.

Let $t_\tau$ be the \textit{measure topology}~\cite{Ne} on
$S(\mathcal{M},\tau)$ whose base of neighbourhoods of zero is given
by
\begin{gather*}
  U(\varepsilon,\delta)=\{x\in S(\mathcal{M},\tau): \ \
  \hbox{there exists a projection} \ \
  p\in \mathcal{P}(\mathcal{M}),
\\
  \hbox{such that}\ \tau(p^\bot)\leq \delta,
  \ xp \in \mathcal{M}, \ \ \|xp\|_\mathcal{M}\leq\varepsilon\}, \ \ \varepsilon>0, \ \delta>0.
\end{gather*}

The pair $(S(\mathcal{M},\tau),t_\tau)$ is a complete metrizable
topological $*$-algebra. Here, the topology $t_\tau$ majorizes the
topology $t(\mathcal{M})$ on $S(\mathcal{M},\tau)$ and, if $\tau$
is a finite trace, the topologies $t_\tau$ and $t(\mathcal{M})$
coincide~\cite{CM}.

\section{Continuity of derivations on a $C^*$-algebra $\mathcal{M}$ with values in a quasi-normed $\mathcal{M}$-bimodule}

Let  $\mathcal{M}$ be a $C^*$-algebra  with identity $\mathbf{1}$ and $X$ be an arbitrary bimodule over  $\mathcal{M}$ with bilinear mappings $(a,x) \rightarrow ax, (a,x) \rightarrow xa : \mathcal{M} \times X \rightarrow X$ such that $\mathbf{1}x=x\mathbf{1}$ for all $x\in X$. By introduction of algebraic operation $\lambda x:=(\lambda \mathbf{1}) x,\lambda\in\mathbb{C},x\in X$, the $\mathcal{M}$-bimodule $X$ become a complex linear space.

Recall that a real function
$\|\cdot\|$ on a complex linear space $X$ is called a \textit{quasi-norm} on $X$, if there exists a constant $ C\geq 1$ such that for all
$x,y\in X,\alpha \in\mathbb{C}$ the following properties hold:

$(i)$. \ \ $\|x\|\geq 0, \|x\|=0 \Leftrightarrow x=0$;

$(ii)$.\ \  $\|\alpha x\|=|\alpha|\|x\|$;

$(iii)$. $\|x+y\|\leq C(\|x\|+\|y\|)$.

The couple $(X,\|\cdot\|)$ is called a quasi-normed space and the least
of all constants $C$  satisfying the  inequality $(iii)$ above is
called the modulus of concavity of the quasi-norm $\|\cdot\|$ and denoted by $C_X$. Every quasi-normed space $(X,\|\cdot\|)$ is a locally bounded Hausdorff topological vector space (see e.g. \cite{Kalton}).

If an $\mathcal{M}$-bimodule $X$ is equipped with a quasi-norm
$\|\cdot \|_X$, then the couple $(X,\|\cdot\|_X)$ is called a quasi-normed $\mathcal{M}$-bimodule if for all $x\in X, a,b\in\mathcal{M}$ the equality
\begin{equation}
\| axb\|_X\leq \| a\|
_{\mathcal{M}}\| b\| _{\mathcal{M}}\|
x\|_X,
\label{ChIVeq21}
\end{equation}
holds, where $\|\cdot\|_\mathcal{M}$ is the $C^*$-norm on $\mathcal{M}$.

Let $X$ be an $\mathcal{M}$-bimodule over a $C^*$-algebra $\mathcal{M}$.
A linear mapping  $\delta:\mathcal{M}\rightarrow X$ is called a \emph{derivation}, if
$\delta(ab)=\delta(a)b+a\delta(b)$ for all $a,b\in\mathcal{M}$. A derivation $\delta:\mathcal{M}\rightarrow X$ is called
\emph{inner}, if there exists an element $d\in X$, such that $\delta(x)=[d,x]=dx-xd$ for all $x\in\mathcal{M}$.

If $(X,\|\cdot\|_X)$ is a quasi-normed $\mathcal{M}$-bimodule, then, by \eqref{ChIVeq21} every inner derivation $\delta:\mathcal{M}\rightarrow X$ is a continuous linear mapping from $(\mathcal{M},\|\cdot\|_\mathcal{M})$ into $(X,\|\cdot\|_X)$.

In the following theorem we show that every derivation  $\delta\colon\mathcal{M}\to X$ is continuous and strengthen the Ringrose Theorem by omitting the assumption of completeness of the space $(X,\|\cdot\|_X)$. Our proof uses the original proof of J.Ringrose \cite[Theorem 2]{Ringrose}.

\begin{theorem}
\label{t_Ring}
Let $(X,\|\cdot\|_X)$ be a quasi-normed $\mathcal{M}$-bimodule. Then every derivation $\delta\colon\mathcal{M}\to X$ is a continuous mapping from $(\mathcal{M},\|\cdot\|_\mathcal{M})$ into $(X,\|\cdot\|_X)$.
\end{theorem}
\begin{proof}
Set $$\mathcal{J}=\{x\in\mathcal{M}| \mbox{ the mapping } u\mapsto\delta(xu) \mbox{ from } \mathcal{M} \mbox{ into } X \mbox{ is continuous } \}.$$
Exactly as in \cite[Theorem 2]{Ringrose} we obtain that $\mathcal{J}$ is a two-sided ideal in $\mathcal{M}$.

Since $\delta(xu)=\delta(x)u+x\delta(u)$ we have
$$\mathcal{J}=\{x\in\mathcal{M}| S_x(u):=x\delta(u)  \mbox{ is continuous map from } \mathcal{M} \mbox{ into } X\}.$$
Let us show that $\mathcal{J}$ is a closed ideal in $\mathcal{M}$. Let $\{x_n\}_{n=1}^\infty\subset \mathcal{J}, x\in\mathcal{M}$ and $\|x_n-x\|_\mathcal{M}\rightarrow 0$ as $n\rightarrow\infty$. Then every $S_{x_n}$ is a continuous mapping from $\mathcal{M}$ into $X$ and for every fixed $u\in\mathcal{M}$ we have
$$\|S_x(u)-S_{x_n}(u)\|_X=\|x\delta(u)-x_n\delta(u)\|_X \leq\|x-x_n\|_\mathcal{M}\|\delta(u)\|_X\rightarrow 0.$$

The principle of uniform boundedness (see e.g. \cite[Theorem 2.8]{Rudin}) implies that $S_x$ is continuous, that is $x\in\mathcal{J}$. Thus, $\mathcal{J}$ is a closed two-sided ideal in $\mathcal{M}$.

Now, we show that the restriction $\delta|_\mathcal{J}$ is continuous. Suppose the contrary. Then there exists a sequence $\{x_n\}_{n=1}^\infty\subset\mathcal{J}$ such that
$$\sum_{n=1}^\infty\|x_n\|_\mathcal{M}^2\leq 1 \mbox{ and } \|\delta(x_n)\|_X\rightarrow\infty.$$
 Set $y=(\sum_{n=1}^\infty x_nx_n^*)^{\frac{1}{4}}$. Then
$y\in\mathcal{J}_+,\|y\|_\mathcal{M}\leq 1 \mbox{ and } x_nx_n^*\leq y^4.$ By \cite[Lemma 1]{Ringrose} there exists $z_n\in\mathcal{J}$ with $\|z_n\|_\mathcal{M}\leq 1$ such that $x_n=yz_n$. Hence, $\|\delta(yz_n)\|_X=\|\delta(x_n)\|_X\rightarrow\infty,$ that is the mapping $u\mapsto\delta(yu)$ is unbounded. That contradicts the inclusion $y\in\mathcal{J}$. Thus, the restriction $\delta|_\mathcal{J}$ is a continuous mapping.

Next, we claim that the quotient algebra $\mathcal{M/J}$ is finite-dimensional. Assume the contrary and as in \cite[Theorem 2]{Ringrose} choose a sequence $y_j\in\mathcal{M}_+$ such that $$\|y_j\|_\mathcal{M}\leq 1, y_j^2\notin\mathcal{J} \mbox{ and } y_jy_i=0 \mbox{ for } j\neq i.$$ Since $y_j^2\notin\mathcal{J}$ the mapping $u\mapsto\delta(y_j^2u),u\in\mathcal{M},$ is unbounded, therefore there exist $u_j\in\mathcal{M}$ such that
$$\|u_j\|_\mathcal{M}\leq 2^{-j}\mbox{ and } \|\delta(y_j^2u_j)\|_X\geq C_X(\|\delta(y_j)\|_X+j).$$ By setting $z=\sum_{j=1}^\infty y_ju_j\in\mathcal{M}$ we have $\|z\|_\mathcal{M}\leq 1$ and $y_jz=y_j^2u_j$.

The inequality $\|x\|_X\leq C_X\|x-y\|_X+C_X\|y\|_X, x,y\in X,$ implies that
\begin{gather*}
\begin{split}
\|y_j\delta(z)\|_X&=\|\delta(y_jz)-\delta(y_j)z\|_X\geq
C_X^{-1}\|\delta(y_j^2u_j)\|_X-\|\delta(y_j)z\|_X
\\
&\geq \|\delta(y_j)\|_X+j-\|\delta(y_j)\|_X\|z\|_\mathcal{M}\geq j.
\end{split}
\end{gather*}

Now, using the inequality $\|y_j\|_\mathcal{M}\leq 1$ we obtain that
$$1\leq j^{-1} \|y_j\delta(z)\|_X\leq j^{-1}\|y_j\|_\mathcal{M}\|\delta(z)\|_X \rightarrow 0 \mbox{ as } j\rightarrow\infty,$$
that is a contradiction, thus the quotient algebra $\mathcal{M/J}$ is finite-dimensional.

Therefore, $\mathcal{M}$ is a direct sum $Y\oplus\mathcal{J}$, where $Y$ is a finite-dimensional subspace in $\mathcal{M}$ and $\mathcal{J}$ is a closed ideal in $(\mathcal{M},\|\cdot\|_\mathcal{M})$. Since the restrictions $\delta|_\mathcal{J}$ and $\delta|_Y$ are continuous it follows that the derivation $\delta:(\mathcal{M},\|\cdot\|_\mathcal{M})\to(X,\|\cdot\|_X)$ is continuous too.
\end{proof}

\section{Description of derivations with values in quasi-normed $\mathcal{M}$-bimodules of locally measurable operators}

In this section  we establish the main result of the present paper, which give description of all derivations acting on a von Neumann algebra $\mathcal{M}$ with values in a quasi-normed $\mathcal{M}$-bimodule $\mathcal{E}$ of locally measurable operators affiliated with $\mathcal{M}$.

Let $\mathcal{M}$ be an arbitrary von Neumann algebra.
A linear subspace
$\mathcal{E}$ of $LS(\mathcal{M})$, is called an
$\mathcal{M}$-bimodule of locally measurable operators if $axb\in
\mathcal{E}$ whenever $x\in \mathcal{E}$ and $a,b\in \mathcal{M}$. It is clear that $\mathcal{E}$ is a bimodule over the $C^*$-algebra $\mathcal{M}$ in the sense of section 3.
If $\mathcal{E}$ is a $\mathcal{M}$-bimodule of locally measurable
operators, $x\in\mathcal{E}$ and $x=v|x|$ is the polar
decomposition of operator $x$ then
$|x|=v^*x\in\mathcal{\mathcal{E}}$ and $x^*=|x|v^*\in\mathcal{E}$.
In addition, by Proposition \ref{mchext} every $\mathcal{M}$-bimodule satisfies the following condition
\begin{equation}
\label{e_solid}\text{if}\ |x|\leq |y|,\ y\in\mathcal{E},\ x\in
LS(\mathcal{M})\ \text{then}\ x\in\mathcal{E}.
\end{equation}

If an $\mathcal{M}$-bimodule of locally measurable operators $\mathcal{E}$ is equipped with a quasi-norm
$\|\cdot \|_{\mathcal{E}}$, satisfying the inequality
\eqref{ChIVeq21},
then $(\mathcal{E}, \|\cdot \|_{\mathcal{E}})$ is called a \emph{quasi-normed $\mathcal{M}$-bimodule of locally measurable operators}.

Examples of
quasi-normed $\mathcal{M}$-bimodules of locally measurable operators, which are not  normed $\mathcal{M}$-bimodules, are  the noncommutative $L_p$-space $L_p(\mathcal{M},\tau)$ associated with a faithful normal semifinite trace $\tau$ for $p\in(0,1)$.

It is easy to see that for the quasi-norm $\|\cdot\|_{\mathcal{E}}$ on a quasi-normed $\mathcal{M}$-bimodule of locally measurable operators $\mathcal{E}$ the following properties hold:
\begin{equation}
\label{e1}
\||x|\|_\mathcal{E}=\|x^*\|_\mathcal{E}=\|x\|_\mathcal{E}\ \text{for any}\ x\in\mathcal{E};
\end{equation}
\begin{equation}
\label{e2}
\|y\|_\mathcal{E}\leq\|x\|_\mathcal{E}\ \text{for any}\ x,y\in\mathcal{E},\ 0\leq y\leq x \mbox{ (see Proposition \ref{mchext})}.
\end{equation}

Recall that two projections $e,f\in\mathcal{P}(\mathcal{M})$ are called equivalent (notation: $e\sim f$) if there exists a partial isometry  $u\in\mathcal{M}$ such that $u^*u=e$ and $uu^*=f$.
For projections $e,f\in\mathcal{P}(\mathcal{M})$ notation $e\preceq f$ means that there exists a projection $q\in\mathcal{P}(\mathcal{M})$ such that
$e\sim q\leq f$.

\begin{proposition}
\label{pe4}
Let  $(\mathcal{E}, \|\cdot \|_{\mathcal{E}})$  be a quasi-normed $\mathcal{M}$-bimodule of locally measurable operators, $p,p_k\subset \mathcal{P}(\mathcal{M})\cap\mathcal{E}, k=1,2,\dots, q\in\mathcal{P}(\mathcal{M}), q\preceq p$. Then $q \in\mathcal{E},\sup\limits_{1\leq k\leq n} p_k\in\mathcal{E}$ and
\begin{equation*}
\|q\|_\mathcal{E}\leq\|p\|_\mathcal{E}, \quad\quad
\bigl\|\sup\limits_{1\leq k\leq n} p_k\bigl\|_\mathcal{E}\leq\sum_{k=1}^n C_\mathcal{E}^k\|p_k\|_\mathcal{E}.
\end{equation*}
\end{proposition}
\begin{proof}
If $q\sim e\leq p$ and $u$ is a partial isometry from $\mathcal{M}$ such that $u^*u=q, uu^*=e$, then $q=u^*eu \in\mathcal{E}$ and $\|q\|_\mathcal{E}\leq\|u^*\|_\mathcal{M}\|u\|_\mathcal{M}\|e\|_\mathcal{E}\leq\|p\|_\mathcal{E}$.

Since $p_1\vee
p_2-p_1\sim p_2-p_1\wedge p_2\leq p_2$ we have that $p_1\vee p_2-p_1\in\mathcal{E}$
and $\|p_1\vee p_2-p_1\|_{\mathcal{E}}\leq\|p_2\|_{\mathcal{E}}$. Hence, $p_1\vee p_2=(p_1\vee p_2-p_1)+p_1\in\mathcal{E}$
and
$$\|p_1\vee p_2\|_{\mathcal{E}}=\|p_1\vee p_2-p_1+p_1\|_\mathcal{E}\leqslant C_\mathcal{E}(\|p_1\vee
p_2-p_1\|_{\mathcal{E}}+\|p_1\|_{\mathcal{E}})\leq C_\mathcal{E}\|p_1\|_\mathcal{E}+C_\mathcal{E}\|p_2\|_\mathcal{E},$$
in particular, $\|p_1\vee p_2\|_{\mathcal{E}}\leq C_\mathcal{E}\|p_1\|_\mathcal{E}+C_\mathcal{E}^2\|p_2\|_\mathcal{E}$ since $C_\mathcal{E}\geq 1$.

Further, proceed by the induction we have that $\sup\limits_{1\leq k\leq n}  p_k\in\mathcal{E}$ and
\begin{gather*}
\bigl\|\sup\limits_{1\leq k\leq n}  p_k\bigl\|_\mathcal{E}=\bigl\|p_1\vee(\sup\limits_{2\leq k\leq n}  p_k)\bigl\|_\mathcal{E}\leq C_\mathcal{E}\|p_1\|_\mathcal{E}+C_\mathcal{E}\bigl\|\sup\limits_{1\leq k\leq n-1} p_{k+1}\bigl\|_\mathcal{E}
\\
\leq C_\mathcal{E}\|p_1\|_\mathcal{E}+C_\mathcal{E}\sum_{k=1}^{n-1} C_\mathcal{E}^k\|p_{k+1}\|_\mathcal{E}=\sum_{k=1}^n C_\mathcal{E}^k\|p_k\|_\mathcal{E}.
\end{gather*}
\end{proof}

Let $\delta:\mathcal{M}\to LS(\mathcal{M})$ be an arbitrary derivation, that is $\delta$ is a linear mapping such that $\delta(ab)=\delta(a)b+a\delta(b)$ for all $a,b\in\mathcal{M}$. In \cite[Lemma 3.1]{BCS_2}  it is proved that $\delta(z)=0$  and $\delta(zx)=z\delta(x)$  for all $x\in\mathcal{M}, z\in\mathcal{P(Z(M))}$. In particular, $\delta(z\mathcal{M})\subset zLS(\mathcal{M})$ and the restriction $\delta^{(z)}$ of the derivation $\delta$ to $z\mathcal{M}$ is a derivation on $z\mathcal{M}$ with values in $zLS(\mathcal{M})=LS(z\mathcal{M})$.

Let $\delta$ be a derivation from $\mathcal{M}$ with values in a $\mathcal{M}$-bimodule $\mathcal{E}$ of locally measurable operators. Let us define a mapping
$$\delta^*: \mathcal{M}\rightarrow
\mathcal{E},$$
by setting $\delta^*(x)=(\delta(x^*))^*$, $x\in
\mathcal{M}$. A direct verification shows that $\delta^*$ is also a
derivation from $\mathcal{M}$ with values in $\mathcal{E}$.

A derivation $\delta$ is said to be
\emph{self-adjoint}, if $\delta=\delta^*$. Every
derivation $\delta$ from $\mathcal{M}$ with values in a  $\mathcal{M}$-bimodule $\mathcal{E}$  of locally measurable operators can be represented in the form
$\delta= Re(\delta)+ i Im(\delta)$, where
$Re(\delta)=(\delta+\delta^*)/2,\ Im(\delta)=(\delta-\delta^*)/2i$
are self-adjoint derivations from $\mathcal{M}$ with values in $\mathcal{E}$.

For the proof of the assertion that every derivation $\delta:\mathcal{M}\to\mathcal{E}$ is inner we need the following significant theorem establishing that every $t(\mathcal{M})$-continuous derivation acting on the $*$-algebra $LS(\mathcal{M})$ is inner.
\begin{theorem}\cite[Theorem 4.1]{BCS_3}
\label{t1}
Every derivation on the algebra $LS(\mathcal{M})$ continuous with respect to the topology $t(\mathcal{M})$ is inner derivation.
\end{theorem}

For application of the Theorem \ref{t1} in our case, when $\delta$ is a derivation acting on $\mathcal{M}$ with values in quasi-normed $\mathcal{M}$-bimodule $(\mathcal{E},\|\cdot\|_\mathcal{E})$ of locally measurable operators, we need $t(\mathcal{M})$-continuity of $\delta$. In the following Proposition \ref{p4} using Theorem \ref{t_Ring} we prove this property of derivation $\delta$.

Let $\tau$ be a faithful normal finite trace $\tau$
on the von Neumann algebra $\mathcal{M}$. In this case, the algebra $\mathcal{M}$ is finite. Moreover,
$LS(\mathcal{M})=S(\mathcal{M})=S(\mathcal{M},\tau)$,
$t(\mathcal{M})=t_\tau$ and $(LS(\mathcal{M}),t(\mathcal{M}))$ is
an $F$-space \cite[\textsection\textsection 3.4,3.5]{MCh}.

The following lemma shows that quasi-normed $\mathcal{M}$-bimodule $(\mathcal{E},\|\cdot\|_\mathcal{E})$ of locally measurable operators  is continuously embedded into $(LS(\mathcal{M}),t(\mathcal{M}))$.
\begin{lemma}
\label{l_e2}
 If $\{a_n\}_{n=1}^\infty\subset \mathcal{E}$ and $\|a_n\|_{\mathcal{E}}\rightarrow 0$,
then $a_n\xrightarrow{t(\mathcal{M})}0$.
\end{lemma}
\begin{proof}
Since the trace $\tau$ is finite, we have that $t(\mathcal{M})=t_\tau$, and therefore it is sufficient to show that $a_n\xrightarrow{t_\tau} 0$. Suppose the contrary. Passing to a subsequence, if necessary, we may choose $\varepsilon,\delta>0$ such that
\begin{eqnarray}
 \label{pr(1)} \tau(\mathbf{1}-E_\varepsilon(|a_n|))>\delta;
\\
 \label{pr(2)} \|a_n\|_\mathcal{E}<(2C_\mathcal{E})^{-n}\varepsilon
\end{eqnarray}
for all $n\in\mathbb{N}$.

Set $$p_n=\mathbf{1}-E_\varepsilon(|a_n|), q_n=\sup_{m\geq n}p_m, q=\inf_{n\geq 1}q_n.$$ Since $\tau$ is a normal finite trace and $\tau(p_n)>\delta$ (see \eqref{pr(1)}), we have that
\begin{equation}\label{(3)}
\tau(q)\geq\delta.
\end{equation}

The inequalities \eqref{pr(2)} and $0\leq\varepsilon p_n\leq|a_n|$ imply that
\begin{equation}\label{(4)}
\|p_n\|_\mathcal{E}\leq\varepsilon^{-1}\|a_n\|_\mathcal{E}<(2C_\mathcal{E})^{-n}.
\end{equation}

Set $$r_{n,s}=\bigl(\bigvee_{m=n}^{n+s}p_m\bigl)\wedge q.$$ It is clear that $r_{n,s}\leq r_{n,s+1}$. Since $\tau(\mathbf{1})<\infty$ the set $\mathcal{P(M)}$ is an orthocomplemented complete modular lattice,  and therefore $\mathcal{P(M)}$ is a continuous geometry \cite{Kap}. In particular,
$$\sup_{s\geq 1}r_{n,s}=\bigl(\bigvee_{m=n}^\infty p_m\bigl)\wedge q=q_n\wedge q=q.$$
Consequently, for all $n\in\mathbb{N}$ there exists an integer $s_n$ such that $\tau(q-r_{n,s_n})<2^{-n}$. Set $e_n=\inf_{m\geq n}r_{m,s_m}$. The sequence of projections $\{e_n\}$ is increasing, moreover, $e_n\leq q$ and
$$\tau(q-e_n)=\tau(q-\inf_{m\geq n}r_{m,s_m})=\tau(\sup_{m\geq n}(q-r_{m,s_m}))<2^{-(n-1)}.$$
Therefore, $e_n\uparrow q$.

By Proposition \ref{pe4} we have
$$\|e_n\|_\mathcal{E}\leq\|r_{n,s_n}\|_\mathcal{E}\leq\bigl\|\bigvee_{k=n}^{n+s_n}p_k\bigl\|_\mathcal{E}\leq \sum_{k=n}^{n+s}C_\mathcal{E}^k\|p_k\|_\mathcal{E}\stackrel{\eqref{(4)}}{<}\sum_{k=n}^{n+s}2^{-k}<2^{-(n-1)}.$$
Since the sequence $e_n$ is increasing, the last inequality implies that $e_n=0$ for all $n\in\mathbb{N}$. Using the convergence  $e_n\uparrow q$ we obtain $q=0$, that contradicts \eqref{(3)}.
\end{proof}

For the proof of the following Proposition \ref{p4} on $t(\mathcal{M})$-continuity of a derivation $\delta\colon\mathcal{M}\to\mathcal{E}$ we need the following lemma from \cite{BCS_3}.

\begin{lemma}\cite[Lemma 6.8]{BCS_3}
\label{l_e4} If $\mathcal{M}$ is a von Neumann algebra with a faithful normal finite trace $\tau$,
$\{a_n\}_{n=1}^\infty\subset LS(\mathcal{M})$ and $a_n\xrightarrow{t(\mathcal{M})}0$, then there exists a subsequence $\{a_{n_k}\}_{k=1}^\infty$ such that $a_{n_k}=b_k+c_k$, where $b_k\in\mathcal{M},\ c_k\in LS(\mathcal{M}),\ k\in\mathbb{N},\ \|b_k\|_{\mathcal{M}}\rightarrow 0$ and $s(|c_k|)\xrightarrow{t(\mathcal{M})}0$.
\end{lemma}

The following proposition is crucial step in the proof that every derivation $\delta:\mathcal{M}\to\mathcal{E}$ is inner.
\begin{proposition}
\label{p4} Let $\mathcal{M}$ be a von Neumann algebra with a faithful normal finite trace $\tau$, let $(\mathcal{E},\|\cdot\|_\mathcal{E})$ be a  quasi-normed $\mathcal{M}$-bimodule of locally measurable operators and let $\delta$ be a derivation on $LS(\mathcal{M})$ such that $\delta(\mathcal{M})\subset \mathcal{E}$. Then the derivation $\delta$ is $t(\mathcal{M})$-continuous.
\end{proposition}
\begin{proof}
Since $(LS(\mathcal{M}),t(\mathcal{M}))$ is an $F$-space for the proof of $t(\mathcal{M})$-continuity of the mapping $\delta$  it is sufficient to show that the graph of the linear operator $\delta$ is closed.

Suppose that the graph of the operator $\delta$ is not closed. Then there exist a sequence $\{a_n\}_{n=1}^\infty\subset LS(\mathcal{M})$ and $0\neq b\in LS(\mathcal{M})$ such that
$a_n\xrightarrow{t(\mathcal{M})}0$ and $\delta(a_n)\xrightarrow{t(\mathcal{M})}b$.

By Lemma \ref{l_e4} passing, if necessary, to a subsequence, we may assume that $a_n=b_n+c_n$, where $b_n\in\mathcal{M},\ c_n\in LS(\mathcal{M}),\ n\in\mathbb{N},\ \|b_n\|_{\mathcal{M}}\rightarrow 0$ and
$s(|c_n|)\xrightarrow{t(\mathcal{M})}0$ as $n\rightarrow\infty$.

Since the restriction $\delta|_{\mathcal{M}}$ of the derivation $\delta$ to the von Neumann algebra $\mathcal{M}$ is a derivation from $\mathcal{M}$ into the quasi-normed $\mathcal{M}$-bimodule $(\mathcal{E},\|\cdot\|_\mathcal{E})$, by Theorem \ref{t_Ring} we have
$\|\delta(b_n)\|_{\mathcal{E}}\rightarrow 0$. Lemma \ref{l_e2} implies that $\delta(b_n)\xrightarrow{t(\mathcal{M})}0$.

By verbatim repetition of the second part of the proof \cite[Lemma 6.9]{BCS_3} we obtain that $\delta(c_n)\xrightarrow{t(\mathcal{M})}0$.

Thus,
$\delta(a_n)=\delta(b_n)+\delta(c_n)\xrightarrow{t(\mathcal{M})}0$, that contradicts to the inequality  $b\neq 0$. Consequently, the operator
$\delta$ has closed graph, therefore $\delta$ is
$t(\mathcal{M})$-continuous.
\end{proof}

For the proof of the main result we also need the following property of the algebra $LS(\mathcal{M})$.
\begin{theorem} \cite[Theorem 1]{BS}
\label{t_bs} Let $\mathcal{M}$ be a von Neumann algebra and $a\in
LS_h(\mathcal{M})$. Then there exist a self-adjoint operator
$c$ in the centre of the $*$-algebra $LS(\mathcal{M})$ and a family
$\{u_\varepsilon\}_{\varepsilon>0}$ of unitary operators from
$\mathcal{M}$ such that
\begin{equation}
\label{e_bs}
|[a,u_\varepsilon]|\geq (1-\varepsilon)|a-c|.
\end{equation}
\end{theorem}

Now, we give the main result of this paper.

\begin{theorem}
\label{t_mbm} Let $\mathcal{M}$ be a von Neumann algebra and let
$(\mathcal{E},\|\cdot\|_\mathcal{E})$ be a quasi-normed $\mathcal{M}$-bimodule of locally
measurable operators. Then any derivation
$\delta:\mathcal{M}\rightarrow \mathcal{E}$ is inner, that is there exists an element $d\in\mathcal{E}$ such that
$\delta(x)=[d,x],x\in\mathcal{M}$, in addition
$\|d\|_{\mathcal{E}}\leq
2C_\mathcal{E}\|\delta\|_{\mathcal{M}\rightarrow\mathcal{E}}$. If
$\delta^*=\delta$ or $\delta^*=-\delta$ then $d$ may be chosen so that
$\|d\|_{\mathcal{E}}\leq\|\delta\|_{\mathcal{M}\rightarrow\mathcal{E}}$.
\end{theorem}
\begin{proof}
Let $\overline{\delta}$ be a  derivation on $LS(\mathcal{M})$ such that  $\overline{\delta}(x)=\delta(x)$
for all $x\in\mathcal{M}$ (see \cite[Theorem 4.8]{BCS_2}).

Choose pairwise orthogonal central projections  $\{z_\infty,z_i\}_{j\in J}$ such that $z_\infty$ + $\sup_{j\in J} z_j=\mathbf{1}$, $z_\infty\mathcal{M}$ is a properly infinite von Neumann algebra and on every von Neumann algebra $z_j\mathcal{M}$ exists a faithful normal finite trace.
By \cite[Theorem 3.3]{BCS_2} the derivation  $\overline{\delta}^{(z_\infty)}:=\overline{\delta}|_{LS(z_\infty\mathcal{M})}:LS(z_\infty\mathcal{M})\rightarrow LS(z_\infty\mathcal{M})$ is $t(z_\infty\mathcal{M})$-continuous. Since the von Neumann algebra $z_j\mathcal{M}$ is finite with a faithful normal finite trace  and for the derivation
$\overline{\delta}^{(z_j)}:=\overline{\delta}|_{LS(z_j\mathcal{M})}:LS(z_j\mathcal{M})\rightarrow LS(z_j\mathcal{M})$ the inclusion $\overline{\delta}^{(z_j)}(z_j\mathcal{M}) = \delta^{(z_j)}(z_j\mathcal{M}) \subset z_j\mathcal{E}$ holds, Proposition \ref{p4} implies that $\overline{\delta}^{(z_j)}$  is also  $t(z_j\mathcal{M})$-continuous for all $j\in J$. Therefore, by \cite[Corollary 2.8]{BCS_2}, the derivation $\overline{\delta}$ is  $t(\mathcal{M})$-continuous.
By Theorem \ref{t1} the derivation $\overline{\delta}$ is inner, that is
 there exists an element  $a\in LS(\mathcal{M})$, such that
$\overline{\delta}(x)=[a,x]$ for all $x\in LS(\mathcal{M})$.
It is clear that
$[a,\mathcal{M}]=\overline{\delta}(\mathcal{M})=\delta(\mathcal{M})\subset
\mathcal{E}$. Let us show that we can choose $d\in\mathcal{E}$ such that $\delta(x)=[d,x]$ for all $x\in\mathcal{M}$.

Let $a_1=Re(a)=(a+a^*)/2,\ a_2=Im(a)=(a-a^*)/2i$. Since $[a^*,x]=-[a,x^*]^*\in\mathcal{E}$ for any $x\in\mathcal{M}$, it follows that $[a_1,x]=[a+a^*,x]/2\in\mathcal{E}$ and
$[a_2,x]=[a-a^*,x]/2i\in\mathcal{E}$ for all $x\in\mathcal{M}$.

By Theorem \ref{t_bs} and taking $\varepsilon=1/2$ in (\ref{e_bs}) we obtain that there exist  self-adjoint operators
$c_1,c_2$ in the centre of the $*$-algebra $LS(\mathcal{M})$
 and unitary operators
$u_1,u_2\in\mathcal{M}$ such that
$$2|[a_i,u_i]|\geq |a_i-c_i|,\ i=1,2.$$ Since
$[a_i,u_i]\in\mathcal{E}$ and $\mathcal{E}$ is
$\mathcal{M}$-bimodule of locally
measurable operators we have that $d_i:=a_i-c_i\in\mathcal{E}$, $i=1,2$ (see \eqref{e_solid}). Therefore $d=d_1+id_2\in\mathcal{E}$.
Since $c_1,c_2$ are central elements from $LS(\mathcal{M})$ it follows that $\delta(x)=[a,x]=[d,x]$ for all $x\in\mathcal{M}$.

Now, suppose that $\delta^*=\delta$. In this case,
$[d+d^*,x]=[d,x]-[d,x^*]^*=\delta(x)-(\delta(x^*))^*=\delta(x)-\delta^*(x)=0$
for any $x\in\mathcal{M}$. Consequently, the operator
$Re(d)=(d+d^*)/2$ commutes with every elements from $\mathcal{M}$, and by Proposition  \ref{p2_2}, $Re(d)$ is a central element in the algebra $LS(\mathcal{M})$. Therefore we may suggest that
$\delta(x)=[d,x],\ x\in\mathcal{M}$, where $d=ia,\ a\in
\mathcal{E}_h$. According to Theorem \ref{t_bs} there exist
$c=c^*$ from the centre of the algebra $LS(\mathcal{M})$ and a family
$\{u_\varepsilon\}_{\varepsilon>0}$ of unitary operators from
$\mathcal{M}$ such that
$$|[a,u_\varepsilon]|\geq (1-\varepsilon)|a-c|.$$
For $b=ia-ic$ and $\varepsilon=1/2$ we have $$|b|=|a-c|\leq
2|[a,u_{1/2}]|=2|[-id,u_{1/2}]|=2[d,u_{1/2}]\in\mathcal{E}.$$
Consequently, $b\in\mathcal{E}$, moreover,
$$\delta(x)=[d,x]=[ia,x]=[b,x]$$
for all $x\in\mathcal{M}$. Since
$$(1-\varepsilon)|b|=(1-\varepsilon)|a-c|\stackrel{(\ref{e_bs})}{\leq} |[a,u_\varepsilon]|=|[d,u_\varepsilon]|=|\delta(u_\varepsilon)|,$$
it follows that
$$(1-\varepsilon)\|b\|_{\mathcal{E}}\stackrel{(\ref{e2})}{\leq} \|\delta(u_\varepsilon)\|_{\mathcal{E}}\leq \|\delta\|_{\mathcal{M}\rightarrow\mathcal{E}}$$
for all $\varepsilon>0$, that implies the inequality
$\|b\|_{\mathcal{E}}\leq
\|\delta\|_{\mathcal{M}\rightarrow\mathcal{E}}$.

If $\delta^*=-\delta$, then taking $Im(d)$ instead of $Re(d)$ and repeating previous proof we obtain that $\delta(x)=[b,x]$, where
$b\in\mathcal{E}$ and $\|b\|_{\mathcal{E}}\leq
\|\delta\|_{\mathcal{M}\rightarrow\mathcal{E}}$.

Now, suppose that $\delta\neq\delta^*$ and $\delta\neq
-\delta^*$. Equality (\ref{e1}) implies that
\begin{gather*}
\begin{split}
\|\delta^*\|_{\mathcal{M}\rightarrow\mathcal{E}}&=\sup\{\|\delta(x^*)^*\|_{\mathcal{E}}:\
\|x\|_{\mathcal{M}}\leq 1\}\\&= \sup\{\|\delta(x)\|_{\mathcal{E}}:\
\|x\|_{\mathcal{M}}\leq
1\}=\|\delta\|_{\mathcal{M}\rightarrow\mathcal{E}}.
\end{split}
\end{gather*}
Consequently,
$$\|Re(\delta)\|_{\mathcal{M}\rightarrow\mathcal{E}}=
2^{-1}\|\delta+\delta^*\|_{\mathcal{M}\rightarrow\mathcal{E}}\leq
C_\mathcal{E}\|\delta\|_{\mathcal{M}\rightarrow\mathcal{E}}.$$
Similarly,
$\|Im(\delta)\|_{\mathcal{M}\rightarrow\mathcal{E}}\leq C_\mathcal{E}
\|\delta\|_{\mathcal{M}\rightarrow\mathcal{E}}$. Since
$(Re(\delta))^*=Re(\delta)$, $(Im(\delta))^*=Im(\delta)$, there exist $d_1,d_2\in\mathcal{E}$, such that
$Re(\delta)(x)=[d_1,x],\ Im(\delta)(x)=[d_2,x]$ for all
$x\in\mathcal{M}$ and
$\|d_i\|_{\mathcal{E}}\leq\|\delta\|_{\mathcal{M}\rightarrow\mathcal{E}}$,
$i=1,2$. Taking $d=d_1+id_2$, we have that $d\in\mathcal{E}$,
$\delta(x)=(Re(\delta)+i\cdot
Im(\delta))(x)=[d_1,x]+i[d_2,x]=[d,x]$ for all $x\in\mathcal{M}$,
in addition $\|d\|_{\mathcal{E}}\leq
2C_\mathcal{E}\|\delta\|_{\mathcal{M}\rightarrow\mathcal{E}}$.
\end{proof}

Now, let us give an application of Theorem \ref{t_mbm} to derivations on $\mathcal{M}$ with values in quasi-normed symmetric spaces.

Let $\mathcal{M}$ be a semifinite von Neumann algebra and let $\tau$
be a faithful normal semifinite trace on $\mathcal{M}$. Let
$S(\mathcal{M},\tau)$ be the $*$-algebra of all $\tau$-measurable
operators affiliated with $\mathcal{M}$.
For each $x\in S(\mathcal{M},\tau)$ and $t>0$ we define the decreasing rearrangement (or generalised singular value function) by setting
\begin{gather*}
\begin{split}
\mu_t(x)&=\inf\{\lambda>0:\tau(E_\lambda^\bot(|x|))\leq
t\}\\&=\inf\{\|x(\mathbf{1}-e)\|_\mathcal{M}:e\in\mathcal{P}(\mathcal{M}),\tau(e)\leq
t\}.
\end{split}
\end{gather*}

Let $\mathcal{E}$ be a linear subspace in $S(\mathcal{M},\tau)$ equipped with a quasi-norm  $\|\cdot\|_{\mathcal{E}}$ satisfying the following condition:
$$\text{If}\ x\in S(\mathcal{M},\tau),\ y\in\mathcal{E}\
\text{and}\ \mu_t(x)\leq\mu_t(y)\ \text{then}\ x\in\mathcal{E}\
\text{and}\ \|x\|_{\mathcal{E}}\leq\|y\|_{\mathcal{E}}.$$ In this
case, the pair $(\mathcal{E},\|\cdot\|_{\mathcal{E}})$ is called
\emph{quasi-normed symmetric spaces of $\tau$-measurable operators}.
It is easy to see that every quasi-normed symmetric space of measurable operators is a
quasi-normed $\mathcal{M}$-bimodule of locally
measurable operators, and therefore Theorem
\ref{t_mbm} implies the following

\begin{corollary}
\label{c_last} Let $(\mathcal{E},\|\cdot\|_{\mathcal{E}})$ be a quasi-normed
symmetric spaces of $\tau$-measurable operators, affiliated
with a semifinite von Neumann algebra $\mathcal{M}$ with a
faithful semifinite normal trace $\tau$. Then any derivation
$\delta:\mathcal{M}\rightarrow\mathcal{E}$ is continuous and
there exists $d\in\mathcal{E}$ such that $\delta(x)=[d,x]$ for all
$x\in\mathcal{M}$ and $\|d\|_{\mathcal{E}}\leq
2C_\mathcal{E}\|\delta\|_{\mathcal{M}\rightarrow\mathcal{E}}.$
\end{corollary}

In particular, Corollary \ref{c_last} implies that every derivation on $\mathcal{M}$ with values in noncommutative $L_p$-spaces $L_p(\mathcal{M},\tau), 0<p\leq \infty,$ is inner.

\end{document}